\documentclass[11pt]{article}

\usepackage{geometry}
\usepackage{amsthm}
\usepackage{amsmath}
\usepackage{amssymb}
\usepackage{mathtools}
\usepackage{mathrsfs}
\usepackage{subfigure}
\usepackage{multicol}
\usepackage{tikz}
\usepackage{comment}

\newcommand{\bb}{\mathbb}
\newcommand{\mc}{\mathcal}

\newcommand{\abs}[1]{\lvert #1 \rvert}

\newcommand{\norm}[1]{\lvert\lvert #1 \rvert\rvert}
\newcommand{\minus}{\,\backslash\,}

\newcommand{\weak}{\rightsquigarrow}
\renewcommand{\tilde}{\widetilde}
\newcommand{\flip}{\overline}

\DeclareMathOperator{\rank}{rank}
\DeclareMathOperator{\corank}{corank}
\DeclareMathOperator{\sign}{sign}

\theoremstyle{definition}

\newtheorem{thm}{Theorem}[section]
\newtheorem{prop}[thm]{Proposition}

\newtheorem{defn}[thm]{Definition}

\numberwithin{equation}{section}

\title{A counterexample to the extension space conjecture for realizable oriented matroids}
\author{Gaku Liu}

\begin{document}
\maketitle

\begin{abstract}
The extension space conjecture of oriented matroid theory states that the space of all one-element, non-loop, non-coloop extensions of a realizable oriented matroid of rank $d$ has the homotopy type of a sphere of dimension $d-1$. We disprove this conjecture by showing the existence of a realizable uniform oriented matroid of high rank and corank 3 with disconnected extension space.
\end{abstract}

\section{Introduction}

Given an oriented matroid $M$, the set of all one-element, non-loop, non-coloop extensions of $M$ has a natural poset structure, and the order complex of this poset is called the \emph{extension space} $\mc E(M)$ of $M$. A long-standing and central open question in oriented matroid theory is the \emph{extension space conjecture}, which states that if $M$ is realizable, then $\mc E(M)$ is homotopy equivalent to a sphere of dimension $\rank(M) - 1$. 

Sturmfels and Ziegler \cite{SZ93} proved this conjecture for a class of oriented matroids which they called \emph{strongly Euclidean} oriented matroids, which includes all oriented matroids of rank at most 3 or corank at most 2. However, Santos \cite{San01} showed that realizable oriented matroids which are not strongly Euclidean exist both in rank 4 and corank 3. Mn\"{e}v and Richter-Gebert \cite{MR93} showed that the conjecture is false if one removes the realizability assumption on $M$; they constructed non-realizable oriented matroids of rank 4 with disconnected extension spaces. In this paper, we disprove the extension space conjecture by showing that there exists a realizable uniform oriented matroid of high rank (possibly around $10^5$) and corank 3 with disconnected extension space.

Geometrically, one can view the extension space conjecture as an attempt to understand how the space of realizable oriented matroids is embedded in the space of all oriented matroids. In particular, if one fixes a realization $X$ of $M$, the oriented matroids which can be realized as  one-element extensions of $X$ form a poset which is isomorphic to the face lattice of the boundary of a polytope, and hence has spherical order complex $\mc E(X)$. It was hoped that $\mc E(X)$ might be a good representation of the topology of the entire extension space of $M$; in fact, the extension space conjecture is equivalent to the claim that $\mc E(X)$ is a strong deformation retract of $\mc E(M)$ \cite{AS02}. However, our result shows that for very large oriented matroids things can be more complicated. It would be interesting to better understand under what conditions and by how much the topologies of $\mc E(X)$ and $\mc E(M)$ can differ.

Before describing the counterexample, we discuss two problems which are also resolved as a result of this counterexample.

\paragraph{Combinatorial Grassmannians}

The extension space conjecture is a special case of a far-reaching conjecture by MacPherson, Mn\"{e}v, and Ziegler \cite{MZ93} on ``combinatorial Grassmannians.'' Any set $\mc S$ of oriented matroids on the same ground set can be turned into a topological space by considering the order complex of the usual poset structure on $\mc S$, as above. If $M$ is an oriented matroid and $\mc S$ is the set of all rank $r$ \emph{strong images} \cite[Sec.\ 7.7]{BLSWZ} of $M$, then the resulting space is known as the \emph{combinatorial Grassmannian} $\mc G(r, M)$. These spaces (especially the case where $M$ is the rank $n$ oriented matroid on $n$ elements, in which case this space is known as the \emph{MacPhersonian}) can be thought of as combinatorial models for the real Grassmannian and play important roles in the theories of \emph{combinatorial differential manifolds} and \emph{matroid bundles}; see \cite{Mac93} and \cite{And99}. The basic problem surrounding these objects is whether or not they are topologically similar to their real counterpart, the real Grassmannian $G(r, \bb R^n)$. MacPherson, Mn\"{e}v, and Ziegler conjectured that if $M$ is realizable of rank $n$, then $\mc G(r, M)$ is homotopy equivalent to $G(r, \bb R^n)$. The special case $r = n-1$ of this conjecture is equivalent, through some minor additional arguments, to the extension space conjecture for $M$. Our result is thus also a counterexample to the combinatorial Grassmannian conjecture.

An important remaining open question is whether the MacPhersonian is homotopy equivalent to the appropriate Grassmannian. A positive answer would have ramifications in the application of combinatorial differential manifolds to the study of smooth manifolds. Our result, however, could be taken as evidence against the conjecture, and at the very least demonstrates the combinatorial subtleties underlying the problem.

\paragraph{The generalized Baues problem}

The extension space conjecture is also a special case of the \emph{generalized Baues problem} in the theory of fiber polytopes \cite{BKS94}. This problem studies general classes of polytopal subdivisions which are ``induced'' by some projection of polytopes; these classes of subdivisions include triangulations of polytopes, zonotopal tilings, and monotone paths on polytopes. Given a polytope and a class of subdivisions, the set of such subdivisions of this polytope form a poset and associated order complex, and thus can be studied as a topological space; this is the goal of the generalized Baues problem. The problem is motivated by the fact that if one restricts this poset to a certain set of subdivisions called \emph{coherent} subdivisions, one obtains the face lattice of a polytope known as the \emph{fiber polytope} \cite{BS92}. The conjecture is that the homotopy type of the space does not change when one includes the non-coherent subdivisions.

This area of research proved very fruitful and led to major developments in the understanding of \emph{flip graphs}, which are certain graphs connecting the finest subdivisions of a polytope. See the survey \cite{Rei99} for an overview and the paper \cite{San06} and book \cite[Chpt.\ 7, 9.1]{DRS} for more recent results. The connection to the extension space conjecture is as follows: Via the Bohne-Dress theorem \cite[Thm.\ 7.32]{Zie95}, the extension space of a realizable oriented matroid $M$ is isomorphic to the space of all non-trivial zonotopal tilings of the zonotope associated to the dual of $M$. The extension space conjecture is equivalent to the ``generalized Baues conjecture for cubes,'' which states that for any zonotope, the aforementioned space is homotopy equivalent to a sphere. After Rambau and Ziegler disproved the most general form of the generalized Baues conjecture \cite{RZ96} and Santos disproved the more particular ``generalized Baues conjecture for simplices'' \cite{San06}, the generalized Baues conjecture for cubes remained as (possibly) the last unresolved case of interest for the problem. Our result disproves this case by giving a three-dimensional zonotope whose space of non-trivial zonotopal tilings is disconnected.

\vspace{12pt}

The counterexample in this paper is based on a vector configuration used by the author in \cite{Liu16} to construct a zonotope whose flip graph of zonotopal tilings is not connected. This vector configuration is formed by taking the set $\{e_i - e_j : 1 \le i < j \le 4\}$, where $e_i$ is the $i$-th standard basis vector, and repeating each vector in the set a large number of times. Call this configuration $E_N$, where $N$ is the number of times each vector is repeated. Let $\tilde{E}_N$ be a configuration obtained by perturbing each vector in $E_N$ by a small random displacement in the span of $E_N$. Our result is the following.

\begin{thm}
For large enough $N$, with probability greater than 0, $\tilde{E}_N$ contains a subconfiguration $E$ such that the oriented matroid dual to the oriented matroid of $E$ has disconnected extension space.
\end{thm}

The strategy of the proof is to show that the flip graph (Section~\ref{sec:flips}) of all uniform one-element extensions of the dual oriented matroid of $\tilde{E}_N$ is disconnected. A feature of the proof is that it uses probabilistic arguments to show the existence of certain elements in the flip graph; the value of $N$ required for these arguments to work is roughly $10^5$. We then use a known trick (Proposition~\ref{graphtospace}) to convert disconnectedness of flip graphs to disconnectedness of entire posets. Unfortunately, the trick only tells us that there is some subconfiguration $E \subseteq \tilde{E}_N$ whose dual oriented matroid has disconnected extension space, and does not tell us what $E$ is.

Section 2 gives the relevant background on oriented matroids. Section 3 is the main proof.

\section{Oriented matroids} \label{sec:om}

We will give a brief overview of oriented matroids. While this overview is self-contained, some familiarity with the basic concepts is helpful. We refer to Bj\"{o}rner et al.\ \cite{BLSWZ} or Richter-Gebert and Ziegler \cite{RZ04} for a more comprehensive treatment.

\subsection{Basic definitions}

Throughout Section~\ref{sec:om}, let $E$ be a finite set. Let $\{+,-,0\}$ be the set of \emph{signs}, and let $\{+,-,0\}^E$ be the set of \emph{sign vectors} on $E$. For $\alpha \in \{+,-,0\}$, define $-\alpha \in \{+,-,0\}$ in the obvious way. For $X \in \{+,-,0\}^E$, define $-X \in \{+,-,0\}^E$ such that $(-X)(e) = -X(e)$ for all $e \in E$. Define a partial order on $\{+,-,0\}$ by $0 < +$ and $0 < -$, and extend this to the product order on $\{+,-,0\}^E$.

An oriented matroid is a pair $(E, \mc L)$ where $\mc L$ is a set of sign vectors on $E$ satisfying certain axioms. We will not use this axiomatic description in this paper, but we include it for completeness:

\begin{defn}
An \emph{oriented matroid} is a pair $M = (E, \mc L)$ where $\mc L \subseteq \{+,-,0\}^E$ such that
\begin{enumerate}
\item $(0,\dotsc,0) \in \mc L$
\item If $X \in \mc L$, then $-X \in \mc L$.
\item If $X$, $Y \in \mc L$, then $X \circ Y \in \mc L$, where
\[
(X \circ Y)(e)  =
\begin{dcases*}
X(e) & if $X(e) \neq 0$ \\
Y(e) & otherwise.
\end{dcases*}
\]
\item If $X$, $Y \in \mc L$ and $e \in E$ such that $\{X(e),Y(e)\} = \{+,-\}$, then there exists $Z \in \mc L$ such that $Z(e) = 0$ and $Z(f) = (X \circ Y)(f)$ whenever $\{X(f),Y(f)\} \neq \{+,-\}$.
\end{enumerate}
The set $\mc L$ is called the set of \emph{covectors} of $M$. A minimal element of $\mc L \minus \{0\}$ (with respect to the above product order) is called a \emph{cocircuit} of $M$, and the set of cocircuits is denoted $\mc C^\ast(M)$. Every covector of $M$ can be written as $X_1 \circ \dotsb \circ X_k$ where $X_1$, \dots, $X_k \in \mc C^\ast(M)$.
\end{defn}

Given an oriented matroid $M = (E, \mc L)$, an element $e \in E$ is a \emph{loop} of $M$ if $X(e) = 0$ for all $X \in \mc L$. An element $e \in E$ is a \emph{coloop} of $M$ if there is some $X \in \mc L$ with $X(e) = +$ and $X(f) = 0$ for all $f \in E \minus \{e\}$. An \emph{independent set} of $M$ is a set $\{e_1,\dotsc,e_k\} \subseteq E$ such that there exist $X_1$, \dots, $X_k \in \mc L$ with
\[
X_i(e_j) =
\begin{dcases*}
+ & if $i=j$ \\
0 & otherwise.
\end{dcases*}
\]
All maximal independent sets of $M$ have the same size, and this size is the \emph{rank} of $M$. The \emph{corank} of $M$ is $\abs{E} - \rank(M)$. An oriented matroid of rank $d$ is \emph{uniform} if all $d$-element subsets of $E$ are independent sets.


For $X$, $Y \in \{+,-,0\}^E$, we write $X \perp Y$ if the set $\{ X(e) \cdot Y(e) : e \in E \}$ is either $\{0\}$ or contains both $+$ and $-$. The next theorem defines \emph{duality} of oriented matroids.

\begin{thm}
For any oriented matroid $M = (E,\mc L)$ of rank $d$, the pair $M^\ast = (E,\mc L^\ast)$ where
\[
\mc L^\ast = \{ X \in \{+,-,0\} : X \perp Y \text{ for all } Y \in \mc L \}
\]
is an oriented matroid of rank $\abs{E} - d$, called the \emph{dual} of $M$. We have $M^{\ast\ast} = M$.
\end{thm}

Finally, if $M = (E, \mc L)$ is an oriented matroid and $A \subseteq E$, the pair $M \vert_A = (A, \mc L_A)$ where
\[
\mc L_A := \{ X \vert_A : X \in \mc L \}
\]
is an oriented matroid called the \emph{restriction} of $M$ to $A$.

\subsection{Topological representation}

For each $e \in E$, let $v_e \in \bb R^d$ be a vector. For each point $x \in \bb R^d$, we obtain a sign vector $X \in \{+,-,0\}^E$ by letting $X(e)$ be the sign of the inner product $\langle x , v_e \rangle$. The set of all such sign vectors is the set of covectors of an oriented matroid, which we call the \emph{oriented matroid of the vector configuration} $\{v_e\}_{e \in E}$. An oriented matroid is \emph{realizable} if it is the oriented matroid of some vector configuration.

Now assume that all of the $v_e$ above are nonzero, and let $S_e$ be the intersection of the hyperplane normal to $v_e$ with the unit sphere $S^{d-1}$. The $S_e$ form a \emph{sphere arrangement} of $(d-2)$-dimensional spheres in $S^{d-1}$, and each $S_e$ is oriented in the following way: $S_e$ separates $S^{d-1}$ into two hemispheres, exactly one of which has points with positive inner product with $v_e$. The \emph{topological representation theorem} says that all oriented matroids arise as topological deformations of such arrangements; we now describe this more precisely.

A \emph{pseudosphere} in $S^{d-1}$ is an image of $\{x \in S^{d-1} : x_d = 0\}$ under a homeomorphism $\phi : S^{d-1} \to S^{d-1}$. A psudosphere $S$ separates $S^{d-1}$ into two regions called \emph{sides}; if we choose one side to be $S^+$ and the other to be $S^-$, then we say that $S$ is \emph{oriented}. A \emph{pseudosphere arrangement} is a collection $\mc A = \{S_e\}_{e \in E}$ of oriented pseudospheres in $S^{d-1}$ such that
\begin{enumerate}
\item For all $A \subseteq E$, the set $S_A := \bigcap_{e \in A} S_e$ is homeomorphic to a sphere or empty.
\item If $A \subseteq E$ and $e \in E$ such that $S_A \not\subseteq S_e$, then $S_A \cap S_e$ is a pseudosphere in $S_A$ with sides $S_A \cap S_e^+$ and $S_A \cap S_e^-$.
\end{enumerate}

Let $\mc A = \{S_e\}_{e \in E}$ be a pseudosphere arrangement in $S^{d-1}$. For each $x \in S^{d-1}$, we obtain a sign vector $X \in \{+,-,0\}^E$ by setting
\[
X(e) =
\begin{dcases*}
+ & if $x \in S_e^+$ \\
- & if $x \in S_e^-$ \\
0 & if $x \in S_e$
\end{dcases*}
\]
Let $\mc L(\mc A)$ be the set of all sign vectors obtained this way along with the 0 sign vector. Call $\mc A$ \emph{essential} if $\bigcap_{e \in E} S_e = \emptyset$. We can now state the topological representation theorem.

\begin{thm}[Folkman-Lawrence \cite{FL78}]
For any essential pseudosphere arrangement $\mc A$ in $S^{d-1}$, $(E, \mc L(\mc A))$ is an oriented matroid of rank $d$. Conversely, every oriented matroid without loops is $(E, \mc L(\mc A))$ for some essential pseudosphere arrangement $\mc A$, and $\mc A$ is unique up to homeomorphisms $\phi : S^{d-1} \to S^{d-1}$.
\end{thm}

For an oriented matroid $M$, we call an essential pseudosphere arrangement $\mc A$ such that $M = (E, \mc L(\mc A))$ a \emph{topological representation} of $M$. If $M$ has rank $d$ and $\mc A = \{S_e\}_{e \in E}$ is a topological representation of $M$, we call any nonempty $S_A$ (where $A \subseteq E$) for which $\dim S_A > d-1-\abs{A}$ a \emph{special pseudosphere} of $\mc A$. $M$ is uniform if and only if $\mc A$ has no special pseudospheres. The cocircuits of $M$ are given by points of $S_A$ where $\dim S_A = 0$.

\subsection{Extensions, liftings, and weak maps} \label{sec:elw}

Let $M = (E, \mc L)$ be an oriented matroid. Let $M' = (E', \mc L')$ be another oriented matroid such that $E' = E \cup \{f\}$ for some $f \notin E$. 
We say that $M'$ is a \emph{one-element extension}, or \emph{extension}, of $M$ if $M = M' \vert_E$; that is,
\[
\mc L = \{ X \vert_E : X \in \mc L' \}.
\]
We say that $M'$ is a \emph{one-element lifting}, or \emph{lifting}, of $M$ if
\[
\mc L = \{ X \vert_E : X \in \mc L', X(f) = 0 \}.
\]
If $M'$ is an extension (or lifting) of $M$, we call it \emph{trivial} if $f$ is a coloop (resp., a loop) of $M'$. The notions of extension and lifting are dual to each other: $M'$ is a (non-trivial) extension of $M$ if and only if $(M')^\ast$ is a (non-trivial) lifting of $M^\ast$. Finally, if $M'$ is a non-trivial extension of $M$, then $\rank(M') = \rank(M)$, and if $M'$ is a non-trivial lifting of $M$, then $\rank(M') = \rank(M) + 1$.

We can understand liftings better using topological representation. Suppose $M'$ is a lifting of $M$, and assume $\rank(M) = d$ and $M$ and $M'$ have no loops. Let $\mc A = \{S_e\}_{e \in E'}$ be a topological representation of $M'$ in $S^d$; by applying an appropriate homeomorphism $\phi: S^d \to S^d$, we may assume $S_f = \{ x \in S^d : x_{d+1} = 0\}$ and $S_f^+ = \{ x \in S^d : x_{d+1} > 0\}$. Let $\mc A^+ = \{S_e \cap S_f^+\}_{e \in E}$.\footnote{Note that $M'$ is determined by $\mc A^+$. In addition, we have $S_e \neq S_f$ for all $e \neq f$, because otherwise, by the definition of a lifting, $e$ would be a loop of $M$.} Consider the ``gnomonic projection'' which maps $S_f^+$ to $\bb R^d$. The image of $\mc A^+$ under this map is a (not necessarily central) arrangement $\mc B$ of oriented pseudohyperplanes in $\bb R^d$ such that the intersection of $\mc B$ with the ``sphere at infinity'' (that is, $S_f$) is a pseudosphere arrangement representing $M$. Conversely, given such a pseudohyperplane arrangement $\mc B$ (with the appropriate definition of ``pseudohyperplane arrangment''), we can uniquely construct a lifting $M'$ of $M$ such that the set of covectors of $M'$ which are positive on $f$ is topologically represented by $\mc B$. Hence, liftings of $M$ are given by pseudohyperplane arrangements in $\bb R^d$ whose intersection with the sphere at infinity are topological representations of $M$.

Given two oriented matroids $M_1 = (E, \mc L_1)$ and $M_2 = (E, \mc L_2)$ on the same ground set $E$, we say that there is a \emph{weak map} $M_1 \weak M_2$ if for every $X_2 \in \mc L_2$, there exists $X_1 \in \mc L_1$ such that $X_1 \ge X_2$. We say that this weak map is \emph{rank-preserving} if $M_1$ and $M_2$ have the same rank. If $M_1 \weak M_2$ is a rank-preserving weak map, then $M_1^\ast \weak M_2^\ast$ is also a (rank-preserving) weak map \cite[Cor.\ 7.7.7]{BLSWZ}.

For any set $\mc S$ of oriented matroids on the same ground set, we obtain a partial order on $\mc S$ by letting $M_1 \ge M_2$ if there is a weak map $M_1 \weak M_2$. We call the set of all non-trivial extensions of an oriented matroid $M$ partially ordered this way the \emph{extension poset} $\mc E(M)$ of $M$. Similarly, we call the poset of all non-trivial liftings of $M$ the \emph{lifting poset} $\mc F(M)$ of $M$. Since all non-trivial extensions of an oriented matroid $M$ have the same rank, we have $\mc E(M) \cong \mc F(M^\ast)$. The extension poset (or lifting poset) has a unique minimal element $\hat 0$, corresponding to extension by a loop (resp., lifting by a coloop). 

Every poset has an associated \emph{order complex}, which is the simplicial complex whose simplices are finite chains of the poset. The \emph{extension space conjecture} claims that for any realizable oriented matroid $M$, the order complex of $\mc E(M) \minus \hat 0$ is homotopy equivalent to a sphere of dimension $\rank(M)-1$. Since an oriented matroid is realizable if and only if its dual is, this is equivalent to saying that the order complex of $\mc F(M) \minus \hat 0$ is homotopy equivalent to a sphere of dimension $\corank(M) - 1$ for any realizable $M$. We will find a realizable $M$ (with corank greater than 1) such that $\mc F(M) \minus \hat 0$ is disconnected.

\subsection{Flips} \label{sec:flips}

To prove disconnectedness of some $\mc F(M) \minus \hat 0$, we will actually only need to look at the maximal elements of $\mc F(M)$. If $M$ is uniform, the maximal elements of $\mc F(M)$ are precisely the uniform liftings of $M$. We will study these uniform liftings through \emph{flips}.\footnote{These are called \emph{mutations} in \cite[Sec.\ 7.3]{BLSWZ}. An equivalent discussion can be found there.}

The following propositions define a flip and its basic properties. They are easy to see from topological representation; we leave their proofs to the reader.

\begin{prop} \label{flipdef}
Let $M = (E, \mc L)$ be a uniform oriented matroid of rank $d$.
Let $D = \{e_1,\dotsc,e_d\}$ be a $d$-element subset of $E$. Let $X_1$, \dots, $X_d \in \mc L$ be cocircuits such that $X_i(e_j) = 0$ for all $i \neq j$. Suppose that
\[
X_1 \vert_{E \minus D} = X_2 \vert_{E \minus D} = \dotsb = X_d \vert_{E \minus D}.
\]
Let $X^0 \in \{+,-,0\}^E$ be the sign vector with
\[
X^0(e) =
\begin{dcases*}
0 & if $e \in D$ \\
X_1(e) & otherwise
\end{dcases*}
\]
 and let $\flip X_1$, \dots, $\flip X_d$ be the sign vectors with
\[
 \flip X_i(e) =
 \begin{dcases*}
 -X_i(e) & if $e \in D$ \\
 X_i(e) & otherwise.
 \end{dcases*}
 \]
Then there are oriented matroids $M^0 = (E,\mc L^0)$ and $\flip M = (E,\flip{\mc L})$ such that
 \begin{align*}
 \mc C^\ast(M^0) &= \mc C^\ast(M) \minus \{\pm X_1, \dotsc, \pm X_d\} \cup \{\pm X^0\} \\
 \mc C^\ast(\flip M) &= \mc C^\ast(M) \minus \{\pm X_1, \dotsc, \pm X_d\} \cup \{\pm \flip X_1, \dotsc, \pm \flip X_d\}.
 \end{align*}
We call $M^0$ a \emph{flip} of $M$, and say that $M^0$ is a flip between $M$ and $\flip M$. We say that that the cocircuits $X_1$, \dots, $X_d$ are \emph{involved} in this flip.
\end{prop}

\begin{prop} \label{fliprelation}
In the situation of Proposition~\ref{flipdef}, the oriented matroid $\flip M$ is uniform, and there are weak maps $M \weak M^0$ and $\flip M \weak M^0$. Moreover, $M$, $\flip M$, and $M^0$ are the only oriented matroids $N$ such that $N \weak M^0$.
\end{prop}

\begin{prop} \label{flipblock}
In the situation of Proposition~\ref{flipdef}, if $X \in \mc L$ is a cocircuit such that for all $1 \le i \le d$, either $X(e_i) = X_i(e_i)$ or $X(e_i) = 0$, then $X \in \{X_1,\dotsc,X_d\}$.
\end{prop}

\begin{proof}
Let $\mc A = \{S_e\}_{e \in E}$ be a topological representation of $M$. The conditions on $X_1$, \dots, $X_d$ imply that $S_{e_1}^{X_1(e_1)}$, \dots, $S_{e_d}^{X_d(e_d)}$ bound a simplicial region of $\mc A$. The only cocircuits which correspond to points in the closure of this region are $X_1$, \dots, $X_d$.
\end{proof}

Now suppose that $M = (E,\mc L)$ is a uniform lifting of a uniform oriented matroid $M_0 = (E_0, \mc L_0)$, where $E = E_0 \cup \{f\}$. Let $D$, $X_1$, \dots, $X_d$, and $\flip M$ be as in Proposition~\ref{flipdef}, and suppose that $\flip M$ is also a lifting of $M_0$. This implies that $f \notin D$ and $X_1(f) = X_2(f) = \dotsb = X_d(f) \neq 0$. Since replacing all of the $X_1$, \dots, $X_d$ with their negatives does not change $\flip M$, we may assume $X_1(f) = X_2(f) = \dotsb = X_d(f) = +$. Along with the original assumptions on the $X_i$, this completely determines the $X_i$. In this case, $\flip M$ is determined by $D$, and we say $D$ is the \emph{support} of the flip between $M$ and $\flip M$.

Given a uniform oriented matroid $M = (E, \mc L)$, let $G(M)$ denote the graph whose vertices are all uniform liftings of $M$ and whose edges are the flips between them. The following is a version of \cite[Lem.\ 3.1]{Rei99}, \cite[Cor.\ 4.3]{San00} applied to $\mc F(M)$.

\begin{prop} \label{graphtospace}
If $G(M)$ is disconnected, then there is some subset $A \subseteq E$ such that $\mc F(M \vert_A) \minus \hat 0$ is disconnected and $\corank(M \vert_A) > 1$. 
\end{prop}

\begin{proof}
For any poset $P$, an \emph{upper ideal} of $P$ is a subposet $I \subseteq P$ such that $x \in I$ and $y > x$ implies $y \in I$. For any $x \in P$, define the upper ideals $I_{\ge x} = \{ y \in P : y \ge x\}$ and $I_{> x} = \{ y \in P : y > x\}$. The following is an easy exercise.

\begin{prop} \label{graphtoposet}
Let $P$ be a finite connected poset, and let $G$ be an upper ideal of $P$ containing all the maximal elements of $P$. Suppose that $I_{> x}$ is connected for any $x \in P \minus G$. Then $G$ is connected.
\end{prop}

Let $G$ be the subposet of $\mc F(M)$ consisting of all uniform liftings of $M$ and the flips between them. By Proposition~\ref{fliprelation}, $G$ is an upper ideal of $\mc F(M)$. If $G$ is disconnected, by Proposition~\ref{graphtoposet} there is some $M' \in \mc F(M) \minus G$ such that $I_{> M'}$ is disconnected. We now use the following.

\begin{prop}
For any non-maximal $M' \in \mc F(M)$, there exist $A_1$, \dots, $A_k \subseteq E$ such that $I_{\ge M'} \cong \mc F(M \vert_{A_1}) \times \dotsb \mc \times \mc F(M \vert_{A_k})$.
\end{prop}

\begin{proof}
Let $\mc A$ be a topological representation of $M'$. Let $S_{A_1}$, \dots, $S_{A_k}$ be all of the special pseudospheres of $\mc A$, where $A_i$ is the maximal set $A \subseteq E \cup \{f\}$ such that $S_A = S_{A_i}$. Since $M$ is uniform, none of these special pseudospheres intersects $S_f$, and hence they are all 0-dimensional and $A_i \subseteq E$ for all $i$. Thus, moving the arrangement $\mc A$ into a more general position (while still representing a lift of $M$) is equivalent to moving each of the subarrangements $\{S_e\}_{e \in A_i \cup \{f\}}$ into more general position; in other words, the map $I_{\ge M'} \to \mc F(M \vert_{A_1}) \times \dotsb \times \mc F(M \vert_{A_k})$ given by $M'' \mapsto (M'' \vert_{A_1 \cup \{f\}}, \dotsc, M'' \vert_{A_k \cup \{f\}})$ is an isomorphism of posets.
\end{proof}

Now, let $M' \in \mc F(M) \minus G$ and let $I_{\ge M'} \cong \mc F(M \vert_{A_1}) \times \dotsb \mc \times \mc F(M \vert_{A_k})$ as in the previous Proposition. We may assume each $\mc F(M \vert_{A_i})$ is non-trivial. Then $I_{> M'}$ is disconnected only if $k = 1$ and $\mc F(M \vert_{A_1}) \minus \hat 0$ is disconnected. In this case, if $\corank(M \vert_{A_1}) = 1$, then $M'$ is a flip, which contradicts $M' \notin G$. This completes the proof.
\end{proof}

Thus, to disprove the extension space conjecture, it suffices to show the following.

\begin{thm} \label{thm:disc}
There is a realizable uniform oriented matroid $M$ for which $G(M)$ is disconnected.
\end{thm}

\section{Main proof}

The main idea will be to define a large realizable uniform oriented matroid of rank 3 and show that one of its liftings is highly ``entangled.'' This entanglement will be achieved by what can be thought of as ``local reorientations'' which are applied on a random subset of its ground set $E$. Formally, we will build this lifting up from many smaller liftings of the braid arrangement of dimension 3. Many of the ideas here originally appeared in the author's paper \cite[Sec.\ 4]{Liu16}, and some of the exposition is rewritten from there.

\subsection{Liftings of the 3-dimensional braid arrangement}

We will work in the \emph{tropical projective space} $\bb {TP}^3 := \bb R^4 / (1,1,1,1)\bb R$, whose points we write as points of $\bb R^4$ modulo the relation $x \sim x + (c,c,c,c)$ for any $c \in \bb R$. We define an inner product on $\bb{TP}^3$ by $\langle x,y \rangle = \langle x',y' \rangle_{\bb R^4}$, where $x' \in \bb R^4$ satsifies $x'_1 + \dotsb + x'_4 = 0$ and the residue of $x'$ in $\bb{TP}^3$ is $x$, and $y'$ is defined similarly. While $\bb {TP}^3$ is isomorphic to $\bb R^3$, it has a more convenient coordinate system for our purposes.

Let $\Gamma_n^k$ denote the set of all ordered $k$-tuples $(i_1,\dotsc,i_k)$ of distinct $i_1$, \dots, $i_k \in [n]$ under the equivalence relation $(i_1,\dotsc,i_k) \sim (i_2,\dotsc,i_k,i_1)$. We will use $(i_1 \dotsb i_k)$ to denote the equivalence class of $(i_1,\dotsc,i_k)$ in $\Gamma_n^k$. We write $-(i_1 \dotsb i_k)$ to denote $(i_k \dotsb i_1)$.

Let $E_0 := \Gamma_4^2$. Let $e_i$ be the $i$-th standard basis vector of $\bb R^4$ mapped to $\bb {TP}^3$, and let $e_{ij} := e_i - e_j$. Let $M_0 = (E_0, \mc L_0)$ be the oriented matroid of the vector configuration $\{ e_{ij} : 1 \le i < j \le 4 \}$, where $e_{ij}$ is indexed by $(ij) \in  E_0$. This oriented matroid is topologically represented by the intersection of the unit 2-sphere in $\bb {TP}^3$ with the \emph{braid arrangement} $\mc B_0 := \{ H_{ij} : 1 \le i < j \le 4 \}$, where $H_{ij}$ is the oriented hyperplane $\{ x \in \bb {TP}^3 : x_i - x_j = 0 \}$ with positive direction $e_{ij}$.

We construct eight specific liftings of $M_0$. For each $\gamma  = (ijk) \in \Gamma_4^3$, let $\mc B_\gamma$ be the hyperplane arrangement in $\bb {TP}^3$ with hyperplanes
\begin{alignat*}{2}
H_{(ij)}^\gamma &= \{ x \in \bb{TP}^3 : x_i - x_j = 1 \} \qquad& H_{(il)}^\gamma &= \{ x \in \bb{TP}^3 : x_i - x_l = 0 \} \\
H_{(jk)}^\gamma &= \{ x \in \bb{TP}^3 : x_j - x_k = 1 \} \qquad& H_{(jl)}^\gamma &= \{ x \in \bb{TP}^3 : x_j - x_l = 0 \} \\
H_{(ki)}^\gamma &= \{ x \in \bb{TP}^3 : x_k - x_i = 1 \} \qquad& H_{(kl)}^\gamma &= \{ x \in \bb{TP}^3 : x_k - x_l = 0 \}
\end{alignat*}
where $\{l\} = [4] \minus \{i,j,k\}$. To orient these hyperplanes, for any distinct $1 \le p,q \le 4$, define
\[
\alpha_{pq} =
\begin{dcases*}
+ & if $p < q$ \\
- & if $p > q$
\end{dcases*}
\]
and orient each $H_{(pq)}^\gamma$ so that the $\alpha_{pq}$ side of $H_{(pq)}^\gamma$ is in the $e_{pq}$ direction. With this orientation, the intersection of $\mc B_\gamma$ with the sphere at infinity is a topological representation of $M_0$. Hence, as discussed in Section~\ref{sec:elw}, there is a unique lifting $M_\gamma = (E_0 \cup \{f\}, \mc L_\gamma)$ of $M_0$ such that $\mc L_\gamma^+ := \{ X \in \mc L_\gamma : X(f) = + \}$ is topologically represented by $\mc B_\gamma$. It is easily checked that $M_\gamma$ is a maximal element of $\mc F(M_0)$.

\begin{prop} \label{Bgamma}
Let $\gamma = (ijk)$ be as above.
\begin{enumerate} \renewcommand{\labelenumi}{(\alph{enumi})}
\item For each $p \in \{i,j,k\}$, there are cocircuits $X_1$, $X_2$, and $X_3 \in \mc L_\gamma^+$ satisfying
\begin{alignat*}{4}
X_1((jk)) &= \alpha_{kj} \qquad& X_1((ki)) &= 0 	       \qquad& X_1((ij)) &= 0 		     \qquad& X_1((pl)) &= 0 \\
X_2((jk)) &= 0 		     \qquad& X_2((ki)) &= \alpha_{ik} \qquad& X_2((ij)) &= 0 		     \qquad& X_2((pl)) &= 0 \\
X_3((jk)) &= 0 		     \qquad& X_3((ki)) &= 0 	        \qquad& X_3((ij)) &= \alpha_{ji} \qquad& X_3((pl)) &= 0
\end{alignat*}
\item There is a cocircuit $X \in \mc L_\gamma^+$ satisfying
\begin{alignat*}{2}
X((jk)) &= \alpha_{kj} \qquad& X((il)) &= 0 \\
X((ki)) &= \alpha_{ik} \qquad& X((jl)) &= 0 \\
X((ij)) &= \alpha_{ji}   \qquad& X((kl)) &= 0
\end{alignat*}
\end{enumerate}
\end{prop}

\begin{proof}
In the arrangement $\mc B_\gamma$, the covector $X_1$ corresponds to the point $x \in \bb{TP}^3$ with $x_i = 0$, $x_j = -1$, $x_k = 1$, and $x_l = x_p$. $X_2$ and $X_3$ can be found similarly. The covector $X$ corresponds to the point $(0,0,0,0)$.
\end{proof}

\subsection{A group action on $\Gamma_4^3$} \label{sec:grpaction}

We will use many copies of the liftings in the previous section to construct a lifting of a larger oriented matroid. To help in doing so, we define a certain group action on $\Gamma_4^3$.

For each $\gamma  = (ijk) \in \Gamma_4^3$, we define a function $o_\gamma : \binom{[4]}{3} \to \Gamma_4^3$ by
\begin{align*}
o_\gamma(\{i,j,k\}) &= (ijk) \\
o_\gamma(\{i,j,l\}) &= (ijl) \\
o_\gamma(\{j,k,l\}) &= (jkl) \\
o_\gamma(\{k,i,l\}) &= (kil)
\end{align*}
where $\{l\} = [4] \minus \{i,j,k\}$. It is easy to check that $\gamma$ is determined by $o_\gamma$.

The relationship of $o_\gamma$ to $M_\gamma$ is as follows. Suppose that $\gamma \in \Gamma_4^3$ and $o_\gamma(\{i,j,k\}) = (ijk)$. Then the restriction of $\mc B_\gamma$ to the hyperplanes $H_{(ij)}^\gamma$, $H_{(jk)}^\gamma$, $H_{(ki)}^\gamma$ is isomorphic to the hyperplane arrangement
\begin{align}
H_{ij}' &= \{ x \in \bb{TP}^3 : x_i - x_j = 1 \} \nonumber \\
H_{jk}' &= \{ x \in \bb{TP}^3 : x_j - x_k = 1 \} \label{triple} \\
H_{ki}' &= \{ x \in \bb{TP}^3 : x_k - x_i = 1 \} \nonumber
\end{align}
where $H_{(ij)}^\gamma$ maps to $H_{ij}'$, etc., and each $H_{pq}'$ is oriented in the same way that $H_{(pq)}^\gamma$ is.

Now, we will map each $\alpha \in \Gamma_4^2$ to a permutation $\pi_\alpha : \Gamma_4^3 \to \Gamma_4^3$. This map is completely determined by the following rules: For any distinct $i$, $j$, $k$, $l \in [4]$, we have
\begin{align*}
\pi_{(ij)}(ijk) &= (jil) \\
\pi_{(kl)}(ijk) &= (ijl).
\end{align*}
Let $G_{\Gamma_4^3}$ be the permutation group of $\Gamma_4^3$ generated by all the $\pi_\alpha$.

\begin{prop} \label{groupaction}
The following are true.
\begin{enumerate} \renewcommand{\labelenumi}{(\alph{enumi})}
\item Every element of $G_{\Gamma_4^3}$ is an involution, and $G_{\Gamma_4^3}$ is abelian and transitive on $\Gamma_4^3$.
\item For $l \in [4]$, let $H_l$ be the subgroup of $G_{\Gamma_4^3}$ generated by $\pi_{(il)}$ for all $i \in [4] \minus \{l\}$. Let $i$,$j$,$k \in [4] \minus \{l\}$ be distinct, and let $\Gamma_4^3(ijk)$ be the set of all $\gamma \in \Gamma_4^3$ such that $o_\gamma(\{i,j,k\}) = (ijk)$. Then $\Gamma_4^3(ijk)$ is an orbit of $H_l$.
\end{enumerate}
\end{prop}

\begin{proof}
Since each $\gamma$ is determined by $o_\gamma$, we can view $G_{\Gamma_4^3}$ as an action on the set of functions $o_\gamma$.
We check that for all distinct $i$, $j$, $k$, $l \in [4]$ and $\gamma \in \Gamma_4^3$, we have
\begin{align*}
o_{\pi_{(ij)}\gamma}(\{i,j,k\}) &= -o_\gamma(\{i,j,k\}) \\
o_{\pi_{(ij)}\gamma}(\{i,j,l\}) &= -o_\gamma(\{i,j,l\}) \\
o_{\pi_{(ij)}\gamma}(\{j,k,l\}) &= o_\gamma(\{j,k,l\}) \\
o_{\pi_{(ij)}\gamma}(\{k,i,l\}) &= o_\gamma(\{k,i,l\}).
\end{align*}
It follows that we can embed $G_{\Gamma_4^3}$ as a subgroup of $\bb Z_2^4$. This implies that every element of $G_{\Gamma_4^3}$ is an involution and $G_{\Gamma_4^3}$ is abelian. It is also easy to check from the above action on the $o_\gamma$ that every element of $\Gamma_4^3$ has orbit of size 8, and hence $G_{\Gamma_4^3}$ is transitive.

From the above action on $o_\gamma$, we see that $H_l$ maps $\Gamma_4^3(ijk)$ to itself and every element of $\Gamma_4^3(ijk)$ has orbit of size 4 under $H_l$. Since $\abs{\Gamma_4^3(ijk)} = 4$, $\Gamma_4^3(ijk)$ is an orbit of $H_l$.
\end{proof}

\subsection{A non-uniform realizable oriented matroid and a lifting}

We will now construct a non-uniform realizable oriented matroid and one of its liftings. Our desired uniform oriented matroid will be obtained by perturbing this matroid.

Let $N$ be a positive integer to be determined later. Let
\[
E = \{ (i,j,r) : 1 \le i, j \le 4, i \neq j, -N \le r \le N \} \ / \ (i,j,r) \sim (j,i,-r).
\]
That is, the element $(i,j,r) \in E$ is identified with $(j,i,-r) \in E$. Let $M = (E,\mc L)$ be the oriented matroid of the vector configuration $\{ v_e \}_{e \in E}$, where
\[
v_{(i,j,r)} = e_{ij} \text{ if } i < j.
\]

We construct a lifting of $M$. First, let $\mc B$ be the hyperplane arrangement $\{H_e\}_{e \in E}$ where
\[
H_{(i,j,r)} = \{ x \in \bb{TP}^3 : x_i - x_j = r \}
\]
and $H_{(i,j,r)}$ is oriented so that the $\alpha_{ij}$ side of $H_{(i,j,r)}$ is in the $e_{ij}$ direction. The intersection of $\mc B$ with the sphere at infinity is a topological representation of $M$, and hence $\mc B$ defines a lifting $M_{\mc B} = (E \cup \{f\}, \mc L_{\mc B})$ of $M$.

Let $Q$ be the set of $x \in \bb{TP}^3$ such that if $ijkl$ is a permutation of $[4]$ such that $x_i \ge x_j \ge x_k \ge x_l$, then $x_i - x_j$, $x_j - x_k$, and $x_k - x_l$ are integers at most $N$. Let $Q^\star$ be the set of $x \in Q$ such that $\abs{x_i - x_j} \le N$ for all $i$, $j \in [n]$. For each $x \in Q$, the set of hyperplanes
\[
\mc B(x) := \{H_{(i,j,x_i-x_j)} : 1 \le i < j \le 4, \abs{x_i - x_j} \le N\}
\]
intersect at $x$. If $x \in Q^\star$, then $\mc B(x)$ is isomorphic to the braid arrangement $\mc B_0$.

We now construct a maximal element of $\mc F(M)$ by deforming the arrangement $\mc B$. For each $e = (i,j,r) \in E$, let $g_e$ be an independent random element of $G_{\Gamma_4^3}$ which is 1 with probability $1/2$ and $\pi_{(ij)}$ with probability $1/2$. For each $x \in Q$, define
\[
\gamma(x) := \left( \prod_{1 \le i < j \le 4} g_{(i,j,x_i-x_j)} \right) (123) \in \Gamma_4^3.
\]
For each $x \in Q$, we have an injective map $\mc B(x) \to \mc B_{\gamma(x)}$ where $H_{(i,j,x_i-x_j)} \mapsto H_{(ij)}^{\gamma(x)}$. The image of this map is a subarrangement of $\mc B_{\gamma(x)}$, and there is a canonical way to shift the hyperplanes of $\mc B(x)$ to obtain an arrangement $\mc B(x)'$ which is isomorphic to this image. Now, we construct a pseudohyperplane arrangement $\mc B'$ from $\mc B$ as follows: For all $x \in Q$, we deform $\mc B(x)$ so that in a small (i.e.\ radius $\ll 1$) open neighborhood $U$ around $x$ we have the arrangement $\mc B(x)'$, and at infinity the arrangement is unchanged. This deformation of $\mc B(x)$ is not necessarily local to $U$ because if the hyperplanes $H_{(i,j,r)}$, $H_{(j,k,s)}$, $H_{(k,i,t)}$ are in $\mc B(x)$, then they intersect at a line in $\mc B(x)$, but their deformations do not mutually intersect in $\mc B(x)'$. To show that $\mc B'$ is well-defined, we need to show that for any three such hyperplanes, the restriction of $\mc B(x)'$ to the deformations of these hyperplanes is the same for any $\mc B(x)$ which contains these hyperplanes.

Suppose that $x^1$, $x^2 \in Q$ and $H_{(i,j,r)}$, $H_{(j,k,s)}$, $H_{(k,i,t)}$ are hyperplanes contained in both $\mc B(x^1)$ and $\mc B(x^2)$. Hence,
\begin{align*}
x_i^1 - x_j^1 = x_i^2 - x_j^2 &= r \\
x_j^1 - x_k^1 = x_j^2 - x_k^2 &= s \\
x_k^1 - x_i^1 = x_k^2 - x_i^2 &= t.
\end{align*}
By Proposition~\ref{groupaction}(b), for any $x \in Q$, $o_{\gamma(x)}(\{i,j,k\})$ depends only on $g_{(i,j,x_i-x_j)}$, $g_{(j,k,x_i-x_j)}$, and $g_{(k,i,x_k-x_i)}$. Hence, $o_{\gamma(x^1)}(\{i,j,k\}) = o_{\gamma(x^2)}(\{i,j,k\})$. From the discussion in Section~\ref{sec:grpaction}, this implies that the restrictions of $\mc B(x^1)'$ and $\mc B(x^2)'$ to the deformations of $H_{(i,j,r)}$, $H_{(j,k,s)}$, $H_{(k,i,t)}$ are isomorphic arrangements, with the canonical isomorphism. Thus, $\mc B'$ is a well-defined pseudohyperplane arrangement.\footnote{A proof which does not use topological arguments can be found in \cite{Liu16}.}

The intersection of $\mc B'$ with the sphere at infinity is a topological representation of $M$, so we obtain a lifting $M' = (E \cup \{f\}, \mc L')$ of $M$ such that $(\mc L')^+ := \{X \in \mc L' : X(f) = +\}$ is topologically represented by $\mc B'$. By construction, $M'$ is maximal in $\mc F(M)$ and $M' \weak M_{\mc B}$. The following proposition follows from Proposition~\ref{Bgamma} and the properties of $\mc B$.

\begin{prop} \label{Bprime}
Let $x \in Q^\star$, and assume $\gamma(x) = (ijk)$. Let $\{l\} = [4] \minus \{i,j,k\}$.
\begin{enumerate} \renewcommand{\labelenumi}{(\alph{enumi})}
\item For all $-N \le r \le N$ with $r \ge x_j - x_k$, and for all $p \in \{i,j,k\}$, there are cocircuits $X_1$, $X_2$, and $X_3 \in (\mc L')^+$ satisfying
\begin{alignat*}{3}
X_1(j,k,r) &= \alpha_{kj} \qquad& X_1(k,i,x_k-x_i) &= 0 	            \qquad& X_1(i,j,x_i-x_j) &= 0 	\\	     
X_2(j,k,r) &= 0 	       \qquad& X_2(k,i,x_k-x_i) &= \alpha_{ik} \qquad& X_2(i,j,x_i-x_j) &= 0 	\\	     
X_3(j,k,r) &= 0 	       \qquad& X_3(k,i,x_k-x_i) &= 0 	            \qquad& X_3(i,j,x_i-x_j) &= \alpha_{ji}
\end{alignat*}
and
\[
X_1(p,l,x_p-x_l) = 0 \qquad X_2(p,l,x_p-x_l) = 0 \qquad X_3(p,l,x_p-x_l) = 0.
\]
\item There is a cocircuit $X \in (\mc L')^+$ satisfying
\begin{alignat*}{2}
X(j,k,x_j-x_k) &= \alpha_{kj} \qquad& X(i,l,x_i-x_l) &= 0 \\
X(k,i,x_k-x_i) &= \alpha_{ik} \qquad& X(j,l,x_j-x_l) &= 0 \\
X(i,j,x_i-x_j) &= \alpha_{ji} \qquad& X(k,l,x_k-x_l) &= 0.
\end{alignat*}
and for any distinct $p$, $q \in [4]$ and any $-N \le u \le N$ with $u \neq x_p - x_q$,
\[
X(p,q,u) = \sign(x_p-x_q-u)\alpha_{pq}.
\]
\end{enumerate}
\end{prop}

\subsection{A uniform realizable oriented matroid}

We now perturb $M$ and $M'$. Let $0 < \delta \ll 1$ be a small real number. Let
\[
\Delta := \{ \delta(e_i + e_j - e_k - e_l) : i, j, k, l \in [4] \text{ are distinct} \}.
\]
For each $u \in \Delta$ and $i \in [4]$, define $u^i \in \{\pm 1\}$ as follows: If $u = \delta(e_i + e_j - e_k - e_l)$, then
\[
u^i = u^j = 1 \qquad u^k = u^l = -1.
\]

For each $e \in E$, let $\eta_e := u_e + \epsilon_e$, where $\epsilon_e$ is a generic element of $\bb{TP}^3$ with $\norm{\epsilon_e} \ll \delta$, and $u_e$ is chosen independently and uniformly at random from $\Delta$. Let $\tilde v_e := v_e + \eta_e$. Let $\tilde M = (E, \tilde{\mc L})$ be the oriented matroid of the configuration $\{\tilde v_e\}_{e \in E}$. Since the $\epsilon_e$ are generic, $\tilde M$ is uniform.

Let $H_e'$ be the deformation of $H_e$ in $\mc B'$. We can tilt each pseudohyperplane $H_e'$ ``near infinity'' to obtain a pseudohyperplane $\tilde {H_e'}$ whose normal vector far away from the origin is $\tilde v_e$. This gives an arrangement $\tilde{\mc B'} = \{\tilde {H_e'}\}_{e \in E}$ whose intersection with the sphere at infinity is a topological representation of $\tilde M$. If $\delta$ is small enough, this can be done so that within a sphere $S$ of radius $100N$ around the origin, the arrangement $\tilde{\mc B'}$ is the same as $\mc B'$. In other words, $\tilde{\mc B'}$ defines a lifting $\tilde{M'} = (E \cup \{f\}, \tilde{\mc L'})$ of $\tilde M$ such that $\tilde{M'} \weak M'$. In particular, we have $(\tilde{\mc L'})^+ \supseteq (\mc L')^+$, where $(\tilde{\mc L'})^+ := \{ X \in \tilde{\mc L'} : X(f) = +\}$.

For any distinct $i$, $j$, $k \in [4]$ and any $-N \le r,s,t \le N$, the pseudohyperplanes $\tilde{H'}_{(i,j,r)}$, $\tilde{H'}_{(j,k,s)}$, and $\tilde{H'}_{(k,i,t)}$ in $\tilde{\mc B'}$ intersect at a point which is far from the origin; i.e., outside of $S$. This point will either be far in the $e_l$ direction or far in the $-e_l$ direction, where $\{l\} = [4] \minus \{i,j,k\}$. The correct direction depends on the arrangement of $H_{(i,j,r)}'$, $H_{(j,k,s)}'$, $H_{(k,i,t)}'$ in $\mc B'$ and the sign
\[
\beta_{ijk}(r,s,t) := \sign \left( \alpha_{ij}u_{(i,j,r)}^l + \alpha_{jk}u_{(j,k,s)}^l + \alpha_{ki}u_{(k,i,t)}^l \right).
\]
More precisely, we have the following.

\begin{prop} \label{peak}
Let $i$, $j$, $k$, $l \in [4]$ be distinct and $-N \le r,s,t \le N$. Suppose that there exists $X_0 \in (\mc L')^+$ such that
\[
X_0(i,j,r) = \alpha_{ji} \qquad X_0(j,k,s) = \alpha_{kj} \qquad X_0(k,i,t) = \alpha_{ik}.
\]
Then there exists a cocircuit $X \in (\tilde{\mc L}')^+$ such that
\[
X(i,j,r) = 0 \qquad X(j,k,s) = 0 \qquad X(k,i,t) = 0
\]
and for all $p \in \{i,j,k\}$ and $-N \le u \le N$,
\[
X(l,p,u) =  \alpha_{lp} \cdot \beta_{ijk}(r,s,t).
\]
\end{prop}

\begin{proof}
The existence of $X_0$ implies that the restriction of $\mc B'$ to $\{H_{(i,j,r)}', H_{(j,k,s)}', H_{(k,i,t)}'\}$ is isomorphic (with the usual isomorphism) to the arrangement \eqref{triple}. Thus, the restriction of $\tilde{\mc B}'$ to $\tilde{H'}_{(i,j,r)}$, $\tilde{H'}_{(j,k,s)}$, and $\tilde{H'}_{(k,i,t)}$ is an arrangement whose behavior away from the origin is the same as the arrangement
\begin{align*}
\tilde {H_{ij}'} &= \{x \in \bb{TP}^3 : \langle x, e_{ij} + \alpha_{ij}\eta_{(i,j,r)} \rangle = 1 \} \\
\tilde {H_{jk}'} &= \{x \in \bb{TP}^3 : \langle x, e_{jk} + \alpha_{jk}\eta_{(j,k,s)} \rangle = 1 \} \\
\tilde {H_{ki}'} &= \{x \in \bb{TP}^3 : \langle x, e_{ki} + \alpha_{ki}\eta_{(k,i,t)} \rangle = 1 \}
\end{align*}
obtained by tilting the hyperplanes in arrangement \eqref{triple}.

Let $x$ be the intersection of $\tilde {H_{ij}'}$, $\tilde {H_{jk}'}$, and $\tilde {H_{ki}'}$. Let $\eta_{(i,j,r)} = \eta_{ij}^l + \eta_{ij}^\perp$, where $\eta_{ij}^l$ is parallel to $e_l$ and $\eta_{ij}^\perp$ is orthogonal to $e_l$. Similarly let $\eta_{(j,k,s)} = \eta_{jk}^l + \eta_{jk}^\perp$ and $\eta_{(k,i,t)} = \eta_{ki}^l + \eta_{ki}^\perp$. The vectors
\[
e_{ij} + \alpha_{ij}\eta_{ij}^\perp,\, e_{jk} + \alpha_{jk}\eta_{jk}^\perp,\, e_{ki} + \alpha_{ki}\eta_{ki}^\perp
\]
lie in the 2-dimensional subspace of $\bb{TP}^3$ orthogonal to $e_l$. Hence, there are $c_1$, $c_2$, $c_3 \in \bb R$ such that
\[
c_1(e_{ij} + \alpha_{ij}\eta_{ij}^\perp) + c_2(e_{jk} + \alpha_{ij}\eta_{jk}^\perp) + c_3(e_{ki} + \alpha_{ki}\eta_{ki}^\perp) = 0.
\]
Since $e_{ij} + e_{jk} + e_{ki} = 0$ and $\norm{\eta_{ij}^\perp}$, $\norm{\eta_{jk}^\perp}$, $\norm{\eta_{ki}^\perp} < 3\delta$, we can choose $c_1$, $c_2$, and $c_3$ so that $\abs{c_i-1} < C\delta$ for all $i$ and some constant $C$ independent of $\delta$.

Now, we have
\begin{align}
c_1 \langle x, e_{ij} + \alpha_{ij}\eta_{(i,j,r)} \rangle + c_2 \langle x, e_{jk} + \alpha_{jk}\eta_{(j,k,s)} \rangle + c_3 \langle x, e_{ki} + \alpha_{ki}\eta_{(k,i,t)} \rangle &= c_1 + c_2 + c_3 \nonumber \\
\implies \ \langle x, c_1 \alpha_{ij} \eta_{ij}^l + c_2 \alpha_{jk} \eta_{jk}^l + c_3 \alpha_{ki} \eta_{ki}^l \rangle &= c_1 + c_2 + c_3 \nonumber \\
\implies \ \langle x, c_1 \alpha_{ij} \eta_{ij}^l + c_2 \alpha_{jk} \eta_{jk}^l + c_3 \alpha_{ki} \eta_{ki}^l \rangle &> 0 \label{c1c2c3}
\end{align}
where the last inequality holds for small enough $\delta$ since $\abs{c_i-1} < C\delta$. By the definition of $\eta_{ij}$, we have $\eta_{ij}^l = \delta u_{(i,j,r)}^l e_l + o(\delta) e_l$ where $o(\delta) \ll \delta$, and similarly for $\eta_{jk}^l$ and $\eta_{ki}^l$. Thus,
\[
c_1 \alpha_{ij} \eta_{ij}^l + c_2 \alpha_{jk} \eta_{jk}^l + c_3 \alpha_{ki} \eta_{ki}^l = \delta( c_1 \alpha_{ij} u_{(i,j,r)}^l + c_2 \alpha_{jk} u_{(j,k,s)}^l + c_3 \alpha_{ki} u_{(k,i,t)}^l ) e_l + o(\delta)e_l.
\]
Since $\abs{c_i-1} < C\delta$ for all $i$, this becomes
\begin{align*}
c_1 \alpha_{ij} \eta_{ij}^l + c_2 \alpha_{jk} \eta_{jk}^l + c_3 \alpha_{ki} \eta_{ki}^l &= \delta( \alpha_{ij} u_{(i,j,r)}^l + \alpha_{jk} u_{(j,k,s)}^l + \alpha_{ki} u_{(k,i,t)}^l ) e_l + O(\delta^2) e_l + o(\delta) e_l \\
&= \delta( \alpha_{ij} u_{(i,j,r)}^l + \alpha_{jk} u_{(j,k,s)}^l + \alpha_{ki} u_{(k,i,t)}^l ) e_l + o(\delta) e_l.
\end{align*}
Hence,
\begin{align*}
\sign \langle x, c_1 \alpha_{ij} \eta_{ij}^l + c_2 \alpha_{jk} \eta_{jk}^l &+ c_3 \alpha_{ki} \eta_{ki}^l \rangle \\
&= \sign \left( \delta( \alpha_{ij} u_{(i,j,r)}^l + \alpha_{jk} u_{(j,k,s)}^l + \alpha_{ki} u_{(k,i,t)}^l ) + o(\delta) \right) \sign \langle x, e_l \rangle \\
&= \beta_{ijk}(r,s,t) \cdot \sign \langle x, e_l \rangle
\end{align*}
since $\abs{\alpha_{ij} u_{(i,j,r)}^l + \alpha_{jk} u_{(j,k,s)}^l + \alpha_{ki} u_{(k,i,t)}^l} \ge 1$. With \eqref{c1c2c3}, we thus have $\sign \langle x, e_l \rangle = \beta_{ijk}(r,s,t)$.

Returning to $\tilde{\mc B'}$, we conclude that $\tilde{H'}_{(i,j,r)}$, $\tilde{H'}_{(j,k,s)}$, and $\tilde{H'}_{(k,i,t)}$ intersect at a point which is outside of $S$ and far in the $\beta_{ijk}(r,s,t) e_l$ direction. The cocircuit $X$ corresponding to this point is the desired cocircuit.
\end{proof}

We make some final definitions before proceeding. For each $x \in Q^\star$ and $l \in [4]$, define
\[
\beta_l(x) := \beta_{ijk}(x_i - x_j, x_j - x_k, x_k - x_i) \text{ where } o_{\gamma(x)}([4] \minus \{l\}) = (ijk).
\]
For each $x \in Q^\star$ and $i \in [4]$, define
\begin{align*}
R_{i,+}(x) &:= \{ x + ke_i \in Q^\star : k \in \bb Z_+ \} \\
R_{i,-}(x) &:= \{ x + ke_i \in Q^\star : k \in \bb Z_- \}.
\end{align*}
Note that at least one of $R_{i,+}(x)$, $R_{i,-}(x)$ has size at least $N/2$. Indeed, if $\abs{R_{i,+}(x)} = m$, then there is some $j \in [4] \minus \{i\}$ such that $x_i + m + 1 - x_j \ge N + 1$, and hence $x_i - x_j \ge N - m$. Similarly if $\abs{R_{i,+}(x)} = n$, then there is some $k \in [4] \minus \{i\}$ such that $x_k - x_i \ge N - n$. Thus $x_k - x_j \ge 2N - m - n$, and since $x_k - x_j \le N$, we obtain $m+n \ge N$.

\subsection{A special set $\Omega$}

We now define a special set $\Omega \subseteq Q^\star$. We show that with positive probability, it satisfies certain density conditions. This will be used to show disconnectedness of $G(\tilde M)$.

Let $\Omega$ be the set of all $x \in Q^\star$ such that if $\gamma(x) = (ijk)$ and $\{l\} = [4] \minus \{i,j,k\}$, then
\[
\abs{R_{i,\beta_i(x)}(x)}, \abs{R_{j,\beta_j(x)}(x)}, \abs{R_{k,\beta_k(x)}(x)} \ge N/2
\]
and
\begin{alignat*}{2}
\sign(\alpha_{il}u_{(i,l,x_i-x_l)}^j) &= \beta_j(x) \qquad& \sign(\alpha_{li}u_{(l,i,x_l-x_i)}^k) &= \beta_k(x) \\
\sign(\alpha_{jl}u_{(j,l,x_j-x_l)}^k) &= \beta_k(x) \qquad& \sign(\alpha_{lj}u_{(l,j,x_l-x_j)}^i) &= \beta_i(x) \\
\sign(\alpha_{kl}u_{(k,l,x_k-x_l)}^i) &= \beta_i(x) \qquad& \sign(\alpha_{lk}u_{(l,k,x_l-x_k)}^j) &= \beta_j(x).
\end{alignat*}

\begin{prop} \label{probabilities}
Let $i$, $j$, $k \in [4]$ be distinct and $-N \le r,s,t \le N$ be integers with $r+s+t = 0$. Let
\[
L = \{x \in Q^\star : x_i - x_j = r, x_j - x_k = s, x_k - x_i = t\}
\]
and suppose that $o_{\gamma(y)}(\{i,j,k\}) = (ijk)$ for some $y \in L$. Then for each $x \in L$, the probability that $\gamma(x) = (ijk)$ and $x \in \Omega$ is at least $1/864$, and these probabilities are mutually independent over $L$.
\end{prop}

\begin{proof}
By Proposition~\ref{groupaction}(b), the value of $o_{\gamma(x)}(\{i,j,k\})$ is the same for all $x \in L$. By assumption, this value is $(ijk)$. Now, for each $x \in L$, let $A(x)$ be the event that $\gamma(x) = (ijk)$ and let $B(x)$ be the event that $x \in \Omega$. By Proposition~\ref{groupaction}(b), whether $A(x)$ happens depends only on $g_{(i,l,x_i-x_l)}$, $g_{(j,l,x_j-x_l)}$, and $g_{(k,l,x_k-x_l)}$, where $l = [4] \minus \{i,j,k\}$. By Proposition~\ref{groupaction}(a)(b), the probability $A(x)$ happens is $1/4$.

Now fix $x \in L$. For each $p \in \{i,j,k\}$, there is some $\beta_p \in \{+,-\}$ such that $\abs{R_{p,\beta_p}} \ge N/2$. Now, the event that
\[
\sign(\alpha_{il}u_{(i,l,x_i-x_l)}^j) = \beta_j \qquad\text{and}\qquad \sign(\alpha_{li}u_{(l,i,x_l-x_i)}^k) = \beta_k
\]
depends only on $u_{(i,l,x_i-x_l)}$, and from the definition of $\Delta$, the probability this event occurs is at least $1/6$ for any choice of $\beta_j$ and $\beta_k$. Hence, the probability that all the events
\begin{alignat*}{2}
\sign(\alpha_{il}u_{(i,l,x_i-x_l)}^j) &= \beta_j \qquad& \sign(\alpha_{li}u_{(l,i,x_l-x_i)}^k) &= \beta_k \\
\sign(\alpha_{jl}u_{(j,l,x_j-x_l)}^k) &= \beta_k \qquad& \sign(\alpha_{lj}u_{(l,j,x_l-x_j)}^i) &= \beta_i \\
\sign(\alpha_{kl}u_{(k,l,x_k-x_l)}^i) &= \beta_i \qquad& \sign(\alpha_{lk}u_{(l,k,x_l-x_k)}^j) &= \beta_j
\end{alignat*}
occur is at least $1/216$. Now, assume that all these events occur. Since $A(x)$ is independent of these events, we can assume that $A(x)$ occurs as well. Then $o_{\gamma(x)}(\{j,k,l\}) = (jkl)$, and
\begin{align*}
\beta_i(x) &= \beta_{jkl}(x_j-x_k,x_k-x_l,x_l-x_i) \\
&= \sign \left( \alpha_{jk} u_{(j,k,s)}^i + \alpha_{kl}u_{(k,l,x_k-x_l)}^i + \alpha_{lj}u_{(l,j,x_l-x_j)}^i \right).
\end{align*}
Since $\sign(\alpha_{kl}u_{(k,l,x_k-x_l)}^i) = \sign(\alpha_{lj}u_{(l,j,x_l-x_j)}^i) = \beta_i$, the right hand side of this expression is $\beta_i$ no matter what $\alpha_{jk} u_{(j,k,s)}^i$ is. Hence $\beta_i(x) = \beta_i$. Similarly $\beta_j(x) = \beta_j$ and $\beta_k(x) = \beta_k$. It follows that $x \in \Omega$. This occurs with probability at least $(1/4)(1/216) = 1/864$.

Finally, the event $A(x) \cap B(x)$ depends only on the independent variables $g_{(i,l,x_i-x_l)}$, $g_{(j,l,x_j-x_l)}$, $g_{(k,l,x_k-x_l)}$, $u_{(i,l,x_i-x_l)}$, $u_{(j,l,x_j-x_l)}$, and $u_{(k,l,x_k-x_l)}$, and these variables are different for every $x \in L$. So these events are mutually independent over $L$.
\end{proof}

\begin{prop} \label{positiveprob}
For large enough $N$, with probability greater than 0, we have the following: $\Omega$ is nonempty, and for every $x \in \Omega$, if $\gamma(x) = (ijk)$, then the sets
\begin{align*}
S_i(x) &:= \{ y \in R_{i,\beta_i(x)}(x) \cap \Omega : \gamma(y) = (jkl) \} \\
S_j(x) &:= \{ y \in R_{j,\beta_j(x)}(x) \cap \Omega : \gamma(y) = (kil) \} \\
S_k(x) &:= \{ y \in R_{k,\beta_k(x)}(x) \cap \Omega : \gamma(y) = (ijl) \}
\end{align*}
are all nonempty.
\end{prop}

\begin{proof}
Suppose $x \in \Omega$. By Proposition~\ref{probabilities}, the probability that $S_i(x)$ is empty is at most
\[
\left( \frac{863}{864} \right)^{\abs{R_{i,\beta_i(x)}(x)}} \le \left( \frac{863}{864} \right)^{N/2}
\]
and similarly for $S_j(x)$ and $S_k(x)$. By the union bound, the probability that at least one of these sets is empty is at most $3(863/864)^{N/2}$. Since $\abs{\Omega} \le \abs{Q^\star} \le (2N+1)^3$, the probability that this happens for at least one $x \in \Omega$ is at most
\[
3 (2N+1)^3 \left( \frac{863}{864} \right)^{N/2}.
\]
Finally, by Proposition~\ref{probabilities}, the probability that $\Omega$ is empty is at most $(863/864)^N$.
For large enough $N$, the number 
\[
3 (2N+1)^3 \left( \frac{863}{864} \right)^{N/2} + \left( \frac{863}{864} \right)^N
\]
is less than 1. So with positive probability, the desired property is satisfied.
\end{proof}

From now on, we will assume that the conclusion of Proposition~\ref{positiveprob} is satisfied.

\subsection{Proof that $G(\tilde{M})$ is disconnected}

We conclude by showing that $G(\tilde M)$ is disconnected. This proves Theorem~\ref{thm:disc} and disproves the extension space conjecture.

In the following, we will reuse the variables $M$ and $\mc L$ for convenience.

\begin{defn} \label{Gdag}
Let $G^\dagger$ denote the set of all uniform liftings $M = (E \cup \{f\}, \mc L)$ of $\tilde M$ which satisfy the following. (As usual, $\mc L^+ := \{X \in \mc L : X(f) = + \}$ and $\{l\} = [4] \minus \{i,j,k\}$.)
\begin{enumerate} \renewcommand{\labelenumi}{(\alph{enumi})}
\item For all $x \in \Omega$ and $i \in [4]$ with $\gamma(x) = (ijk)$, for all $y \in \Omega$ with $y_j - y_k \ge x_j - x_k$, and for all $p \in \{i,j,k\}$, there exist cocircuits $X_1$, $X_2$, $X_3 \in \mc L^+$ satisfying
\begin{alignat*}{3}
X_1(j,k,y_j-y_k) &= \alpha_{kj} \qquad& X_1(k,i,x_k-x_i) &= 0 	            \qquad& X_1(i,j,x_i-x_j) &= 0 \\
X_2(j,k,y_j-y_k) &= 0 	       \qquad& X_2(k,i,x_k-x_i) &= \alpha_{ik} \qquad& X_2(i,j,x_i-x_j) &= 0 \\
X_3(j,k,y_j-y_k) &= 0                \qquad& X_3(k,i,x_k-x_i) &= 0 		    \qquad& X_3(i,j,x_i-x_j) &= \alpha_{ji}
\end{alignat*}
and
\[
X_1(p,l,x_p-x_l) = 0 \qquad X_2(p,l,x_p-x_l) = 0 \qquad X_3(p,l,x_p-x_l) = 0.
\]
\item For all $x \in \Omega$ with $\gamma(x) = (ijk)$, there exists a cocircuit $X \in \mc L^+$ satisfying
\[
X(i,l,x_i-x_l) = 0 \qquad X(j,l,x_j-x_l) = 0 \qquad X(k,l,x_k-x_l) = 0
\]
in addition to the following:
\begin{enumerate} \renewcommand{\labelenumii}{(\roman{enumii})}
\item For all $y \in \Omega$ with $y_j - y_k \ge x_j - x_k$, $X(j,k,y_j-y_k) = \alpha_{kj}$.
\item For all $y \in \Omega$ with $y_k - y_i \ge x_k - x_i$, $X(k,i,y_k-y_i) = \alpha_{ik}$.
\item For all $y \in \Omega$ with $y_i - y_j \ge x_i - x_j$, $X(i,j,y_i-y_j) = \alpha_{ji}$.
\end{enumerate}
\item For all $x \in \Omega$ and $i \in [4]$ with $\gamma(x) = (ijk)$, and for all $y \in \Omega$ with $y_j - y_k \ge x_j - x_k$, there exists a cocircuit $X \in \mc L^+$ satisfying
\[
X(j,k,y_j-y_k) = 0 \qquad X(k,l,x_k-x_l) = 0 \qquad X(l,j,x_l-x_j) = 0
\]
and for all $z \in S_i(x)$ and $p \in \{j,k,l\}$,
\[
X(i,p,z_i-z_p) = \alpha_{ip} \cdot \beta_i(x).
\]
\end{enumerate}
\end{defn}

\begin{prop} \label{Gnonempty}
$\tilde{M'} \in G^\dagger$.
\end{prop}

\begin{proof}
Property (a) follows from Proposition~\ref{Bprime}(a) and the fact that $(\tilde{\mc L'})^+ \supseteq (\mc L')^+$. Property (b) follows from Proposition~\ref{Bprime}(b).

We now prove (c). Suppose $x \in \Omega$, $i \in [4]$, and $y \in \Omega$ such that $\gamma(x) = (ijk)$ and $y_j - y_k \ge x_j - x_k$. First, note there is a covector $X_0 \in (\tilde{\mc L'})^+$ such that
\[
X_0(j,k,y_j-y_k) = \alpha_{kj} \qquad X_0(k,l,x_k-x_l) = \alpha_{lk} \qquad X_0(l,j,x_l-x_j) = \alpha_{jl}.
\]
Indeed, if $y = x$, then this holds because $o_{\gamma(x)}(\{j,k,l\}) = (jkl)$, and if $y \neq x$, this holds because $y_j - y_k > x_j - x_k$ and the properties of $\mc B$. Thus, by Proposition~\ref{peak}, there exists $X \in (\tilde{\mc L'})^+$ such that
\[
X(j,k,y_j-y_k) = 0 \qquad X(k,l,x_k-x_l) = 0 \qquad X(l,j,x_l-x_j) = 0
\]
and for all $z \in S_i(x)$ and $p \in \{j,k,l\}$,
\[
X(i,p,z_i-z_p) = \alpha_{ip} \cdot \beta_{jkl}(y_j-y_k,x_k-x_l,x_l-x_j).
\]
We have
\[
\beta_{jkl}(y_j-y_k,x_k-x_l,x_l-x_j) = \sign \left( \alpha_{jk}u_{(j,k,y_j-y_k)}^i + \alpha_{kl}u_{(k,l,x_k-x_l)}^i + \alpha_{lj}u_{(l,j,x_l-x_j)}^i \right).
\]
By the definition of $\Omega$, we have $\sign(\alpha_{kl}u_{(k,l,x_k-x_l)}^i) = \sign(\alpha_{lj}u_{(l,j,x_l-x_j)}^i) = \beta_i(x)$. So the right hand side of this expression is $\beta_i(x)$. Hence,
\[
X(i,p,z_i-z_p) = \alpha_{ip} \cdot \beta_i(x)
\]
as desired.
\end{proof}

\begin{prop} \label{Gclosed}
Suppose $M = (E \cup \{f\}, \mc L)$ and $\flip M = (E \cup \{f\}, \flip{\mc L})$ are uniform liftings of $\tilde{M}$ and there is a flip $M^0$ between $M$ and $\flip M$. If $M \in G^\dagger$, then $\flip M \in G^\dagger$.
\end{prop}

\begin{proof}
We need to show that Definition~\ref{Gdag}(a)-(c) hold for $\flip M$.

\emph{Proof of (a).} Suppose there is some $x \in \Omega$, $i \in [4]$, $y \in \Omega$, and $p \in \{i,j,k\}$ such that $\gamma(x) = (ijk)$, $y_j-y_k \ge x_j-x_k$, and $\flip{\mc L}^+$ does not contain three cocircuits satisfying (a). Since $\mc L^+$ does contain such cocircuits $X_1$, $X_2$, and $X_3$, the flip $M^0$ must have support
\[
\{ (j,k,y_j-y_k), (k,i,x_k-x_i), (i,j,x_i-x_j), (p,l,x_p-x_l)\}
\]
and this flip involves $X_1$, $X_2$, and $X_3$. However, by (b), there exists $X \in \mc L^+$ with $X(j,k,y_j-y_k) = \alpha_{kj}$, $X(k,i,x_k-x_i) = \alpha_{ik}$,\footnote{This equality is obtained by taking $y = x$ in (b)(ii).} $X(i,j,x_i-x_j) = \alpha_{ji}$, and $X(p,l,x_p-x_l) = 0$. This contradicts Proposition~\ref{flipblock}, proving (a).

\emph{Proof of (b).} Suppose there is some $x \in \Omega$ with $\gamma(x) = (ijk)$ such that $\flip{\mc L}^+$ does not contain a cocircuit satisfying (b). Since $\mc L^+$ does contain such a cocircuit $X$, we must have one of the following:
\begin{enumerate} \renewcommand{\labelenumi}{(\roman{enumi})}
\item For some $y \in \Omega$ with $y_j - y_k \ge x_j - x_k$, $M^0$ has support
\[
\{ (j,k,y_j-y_k), (i,l,x_i-x_l), (j,l,x_j-x_l), (k,l,x_k-x_l) \}.
\]
\item For some $y \in \Omega$ with $y_k - y_i \ge x_k - x_i$, $M^0$ has support
\[
\{ (k,i,y_k-y_i), (i,l,x_i-x_l), (j,l,x_j-x_l), (k,l,x_k-x_l) \}.
\]
\item For some $y \in \Omega$ with $y_i - y_j \ge x_i - x_j$, $M^0$ has support
\[
\{ (i,j,y_i-y_j), (i,l,x_i-x_l), (j,l,x_j-x_l), (k,l,x_k-x_l) \}.
\]
\end{enumerate}
We will assume case (i); the other cases are analogous. The cocircuit $X$ must be involved in the flip $M^0$. By (c), there is a cocircuit $Y \in \mc L^+$ satisfying $Y(j,k,y_j-y_k) = 0$, $Y(k,l,x_k-x_l) = 0$, and $Y(j,l,x_j-x_l) = 0$, and hence it is also involved in this flip. Suppose $z \in S_i(x)$. If $\beta_i(x) = +$, then $z_i - z_j \ge x_i - x_j$ by definition of $S_i(x)$. Hence $X(i,j,z_i-z_j) = \alpha_{ji}$ by definition of $X$. But $Y(i,j,z_i-z_j) = \alpha_{ij}$ by definition of $Y$. This contradicts the fact that the cocircuits involved in a flip agree outside of the flip's support. Similarly, if $\beta_i(x) = -$, then $X(k,i,z_k-z_i) = \alpha_{ik}$ and $Y(k,i,z_k-z_i) = \alpha_{ki}$, a contradiction. Since $S_i(x)$ is nonempty, we have a contradiction, proving (b).

\emph{Proof of (c).} Suppose there is some $x \in \Omega$, $i \in [4]$, and $y \in \Omega$ such that $\gamma(x) = (ijk)$, $y_j - y_k \ge x_j - x_k$, and $\flip{\mc L}^+$ does not contain a cocircuit satisfying (c). Since $\mc L^+$ does contain such a cocircuit $X$, there must be some $z \in S_i(x)$ and $p \in \{j,k,l\}$ such that $M^0$ has support
\[
\{ (j,k,y_j-y_k), (k,l,x_k-x_l), (l,j,x_l-x_j), (i,p,z_i-z_p) \}.
\]
Since $z \in S_i(x)$, we have $x_k - x_l = z_k - z_l$ and $x_l - x_j = z_l - z_j$. So we can rewrite the support of the flip as
\[
\{ (j,k,y_j-y_k), (k,l,z_k-z_l), (l,j,z_l-z_j), (i,p,z_i-z_p) \}.
\]
Moreover, we have $y_j - y_k \ge x_j - x_k = z_j - z_k$, and since $z \in S_i(x)$, we have $\gamma(z) = (ljk)$. Thus, by (a) with $x$ replaced by $z$ and $i$ replaced by $l$, there are cocircuits $X_1$, $X_2$, $X_3 \in \mc L^+$ such that
\begin{alignat*}{3}
X_1(j,k,y_j-y_k) &= \alpha_{kj} \qquad& X_1(k,l,z_k-z_l) &= 0 	            \qquad& X_1(l,j,z_l-z_j) &= 0 \\
X_2(j,k,y_j-y_k) &= 0 	       \qquad& X_2(k,l,z_k-z_l) &= \alpha_{lk} \qquad& X_2(l,j,z_l-z_j) &= 0 \\
X_3(j,k,y_j-y_k) &= 0                \qquad& X_3(k,l,z_k-z_l) &= 0 		    \qquad& X_3(l,j,z_l-z_j) &= \alpha_{jl}
\end{alignat*}
and
\[
X_1(i,p,z_i-z_p) = 0 \qquad X_2(i,p,z_i-z_p) = 0 \qquad X_3(i,p,z_i-z_p) = 0.
\]
Thus, the flip $M_0$ must involve these three circuits. But we showed in the proof of (a) that there cannot be a flip involving these circuits. This proves (c).
\end{proof}

Thus, $G^\dagger$ is a nonempty connected component of $G(\tilde M)$. Since $\Omega$ is nonempty, it is easy to see that $G^\dagger$ is not all of $G(\tilde M)$. Thus $G(\tilde M)$ is disconnected.

\end{document}